\documentclass{amsart}
\usepackage[utf8]{inputenc}
\usepackage[top=1in, bottom=1.5in, left=1.5in, right=1.3in]{geometry}

\usepackage{amsaddr}
\usepackage{amsmath}
\usepackage{amssymb}
\usepackage{comment}
\usepackage[shortlabels]{enumitem}

\usepackage[dvipsnames]{xcolor}
\definecolor{TUMblue}{HTML}{0065bd}
\usepackage[colorlinks=true,linkcolor=TUMblue,citecolor=OliveGreen,urlcolor=Sepia,linktocpage=true]{hyperref}
\usepackage{tgtermes}

\newtheorem*{lemma*}{Lemma}

\newcommand{\mbb}{\mathbb}

\newcommand{\mc}{\mathcal}

\newcommand{\C}{{\mbb C}}
\newcommand{\R}{{\mbb R}}
\newcommand{\K}{{\mbb K}}
\newcommand{\Q}{{\mbb Q}}

\newcommand{\miniskip}{\kern 0.1em}

\begin{document}

\title{Computing Common Eigenvectors and Simultaneous Triangulation}
\author{Emanuel Malvetti}
\address{School of Natural Sciences, Technische Universit\"at M\"unchen, 85737 Garching, Germany, and Munich Centre for Quantum Science and Technology (MCQST) \& Munich Quantum Valley (MQV)}

\begin{abstract}
We propose an efficient algorithm for computing a common eigenvector of a finite set of square matrices. 
As an immediate consequence we obtain an algorithm for determining whether the matrices admit a simultaneous triangulation, and, if so, for computing a corresponding basis. \medskip

 \smallskip

\end{abstract}

\maketitle

\section{Introduction}

Given a finite set of matrices of size $n\times n$ with entries in an algebraically closed field (such as the complex numbers), our goal is to find a basis in which all matrices take on an (upper) triangular shape, or to conclude that no such basis exists.
The answer to this question is, for instance, relevant in determining stabilizability of certain control systems~\cite{Shorten98,LindbladReduced23}.
Many mathematical characterizations of the simultaneous triangulability of matrices are collected in~\cite{Radjavi00}, and of course the famous theorems of Lie and Engel give further conditions~\cite{Knapp02}.
It is beneficial to take a more abstract point of view.
%The matrices are simultaneously triangular if and only if the first $s$ basis vectors span a common invariant subspace of the matrices for all $0<s<n$.
The given matrices are simultaneously triangulable if and only if they admit a complete flag\miniskip\footnote{A \emph{flag} in a finite-dimensional vector space is a nested sequence of strict subspaces, and it is \emph{complete} if it contains a subspace of each dimension.} of common invariant subspaces.
An important result states that if a complete flag of common invariant subspaces exists, then every flag of common invariant subspaces is contained in a complete one, cf.~\cite[Lem.~1.5.2]{Radjavi00}.
This is, in fact, a special case of the Jordan--Hölder Theorem, which holds for arbitrary modules and groups~\cite{Hazewinkel04,Isaacs09}.
Consequently, in order to find a simultaneous triangulation, it suffices to repeatedly compute (proper, non-trivial) common invariant subspaces.
In fact it suffices to be able to compute one-dimensional invariant subspaces, which of course correspond to (the span of) a common eigenvector.
This will be the main focus of this paper.
Surprisingly the literature on finding common eigenvectors is rather slim.
Often only a ``brute-force'' algorithm is provided~\cite{Dubi09,AlDweik22a}, which however suffers from combinatorial explosion.
A seminal result is given by Shemesh~\cite{Shemesh84} yielding a way to determine the existence of and compute a common eigenvector of two matrices.
An extension to an arbitrary number of matrices is presented in~\cite{Jamiolkowski15}, which is however rather inefficient.
Our aim is to provide an efficient algorithm for computing a simultaneous eigenvector of an arbitrary family of matrices, and as a result we obtain an efficient algorithm to compute a simultaneous triangulation, (or to conclude that no such eigenvector or triangulation exists).

More generally, if a simultaneous triangulation does not exist, it might still be of interest to find the finest block triangulation, i.e.\ a maximal flag of common invariant subspaces.
For this, an algorithm for finding any common invariant subspace is sufficient for the same reason as above.
These problems have recently been considered in~\cite{AlDweik22a,AlDweik23b}.
Interestingly the algorithm given there requires an algorithm for finding a common eigenvector, so that our results also improve methods for the more general problem of finding simultaneous block triangulations.

We focus mostly on numerical linear algebra methods and computations, since our main field of interest is that of the complex numbers, but our results are formulated for general (algebraically closed) fields.
In computational group theory a similar task is that of testing a module over a finite-dimensional matrix algebra (e.g.~a group representation) over a finite field for irreducibility. 
The tool of choice is called the Meataxe algorithm~\cite{Holt05}, and implementations can be found in~\cite{MAGMA,GAP4}.
Numerical and computer algebra algorithms often have quite different applications and are difficult to compare, and we will refrain from doing so here.
However we will point out some caveats for dealing with arbitrary fields along the way.

\subsection*{A Note on Computational Complexity}

We formulate most results for an arbitrary field $\K$, when necessary restricting to algebraically closed fields.
When it comes to determining the (worst-case) time-complexity of our algorithms, we assume that all field operations take constant time.
When performing exact arithmetic, this is an accurate assumption for finite fields, but when working over the rationals $\Q$ this might break down.
For infinite fields like the reals $\R$ or complex numbers $\C$, it is often sufficient to implement the algorithms numerically, i.e.~using floating point arithmetic.
In particular when determining eigenvalues, numerical algorithms are preferable, and for most linear algebra tasks numerical algorithms perform well and are stable~\cite{Demmel07}.
Again it is accurate to assume that all field operations take constant time.
Note also that for the multiplication of $n\times n$ matrices we use the standard algorithm with complexity $\mc O(n^3)$, although better complexities can be achieved.

%\begin{comment}

\section{Computing a Common Eigenvector}

Consider a list $A_1,\ldots,A_k$ of $n\times n$ matrices over an algebraically closed field $\mbb K$. 
Our goal is to compute a simultaneous eigenvector of these matrices, if it exists. 
More precisely, we are looking for a vector $v\in\mbb K^n$ such that there are numbers $\lambda_i\in\mbb K$ satisfying $A_i v = \lambda_i v$ for all $i\in\{1,\ldots,k\}$.
Of course, if such numbers $\lambda_i$ are known, it is easy to find $v$ as the solution of a linear system of equations. 
However, even if we know all eigenvalues of all $A_i$, there are up to $n^k$ combinations to try.
This brute-force algorithm, given in~\cite{Dubi09,AlDweik22a}, becomes very inefficient when $k$ is large.
We are looking for an algorithm that has polynomial runtime in $n$ and $k$.

\subsection*{Algorithm for the Commuting Case} 

Let us first consider the special case of commuting matrices $A_1,\ldots,A_k$, that is, $[A_i,A_j]:=A_iA_j-A_jA_i=0$ for all $i,j\in\{1,\ldots,k\}$. 
%The key property used here is that commuting matrices preserve each others eigenspaces
%Note that since $\mbb K$ is algebraically closed, there always exists at least one eigenvalue and corresponding eigenspace.
To find a common eigenvector, we first compute the eigenspaces of $A_1$ and denote the smallest of them by $S_1$. 
Note that since $\mbb K$ is algebraically closed, there always exists a non-trivial eigenspace.
Moreover, as all $A_i$ commute, $S_1$ is a common invariant subspace.\miniskip\footnote{Indeed, if $[A,B]=0$ and $Av=\lambda v$, then $A(Bv)=BAv=\lambda (Bv)$.}
Hence all matrices $A_i$ can be restricted to $S_1$, and the restrictions still commute.
Now consider the restriction $A_2|_{S_1}$ and let $S_2$ be its smallest eigenspace. 
Then $S_2$ is again a common invariant subspace of all $A_i$. 
Continuing in this fashion we obtain a nested sequence $\mbb K^n\supseteq S_1\supseteq S_2\supseteq\ldots\supseteq S_k$ of common invariant subspaces. 
If some $S_i$ is one dimensional, then any non-zero vector in this $S_i$ is a simultaneous eigenvector of all $A_j$, otherwise any non-zero vector in $S_k$ is a simultaneous eigenvector of all $A_i$. 

Assuming $\mbb K=\C$, an approximate numerical algorithm for this task can be implemented with time-complexity $\mc O(n^3)$, since the Schur decomposition can be performed in $\mc O(n^3)$ (e.g.~using the QR algorithm, cf.~\cite{Francis62,Arbenz16}), and in each iteration the size of the subspace is at least halved (unless some restricted matrix is a multiple of the identity, but this can be checked quickly without affecting the runtime). 
Over general fields one would first find a root of the characteristic polynomial (possibly over a field extension if the field is not algebraically closed) and then compute the corresponding eigenspace. 
Such algorithms exist for many different kinds of fields, see for instance~\cite[Sec.~24.8.1]{MAGMA-Handbook}.

\medskip

With this special case solved, we can consider the general case of arbitrary matrices $A_i$. Here the idea is to find a common invariant subspace on which all $A_i$ commute.

\begin{lemma*} \label{lemma:abel-submod}
Let $A_1,\ldots,A_k$ of be a list of $n\times n$ matrices over any field $\mathbb K$ and let
\begin{align*}
T=\bigcap_{i,j=1}^k\ker[A_i,A_j].
\end{align*}
Then it holds that 
\begin{enumerate}[(i)]
\item \label{it:ev-in-T} every common eigenvector of all $A_i$ lies in $T$, and
\item \label{it:T-invariant} if the $A_i$ span a %Lie subalgebra of $\mf{gl}(n,\mbb K)$
Lie algebra\miniskip\footnote{Concretely this means that every commutator can be written as a linear combination: $[A_i,A_j]=\sum_{l=1}^k c_{ij}^l A_l$.}, then $T$ is a common invariant subspace of all $A_i$.
\end{enumerate}
\end{lemma*}

\begin{proof}
\ref{it:ev-in-T}: Let $v$ be a common eigenvector of all $A_i$ with $A_i v = \lambda_i v$, then $[A_i,A_j]v=(\lambda_i\lambda_j-\lambda_j\lambda_i)v=0$, and hence $v\in T$. 
\ref{it:T-invariant}: Now assume that the $A_i$ span a Lie algebra, meaning that there are constants $c_{ij}^l\in\mbb K$ such that $[A_i,A_j]=\sum_{l=1}^k c_{ij}^l A_l$. It is clear that $T$ is a subspace, hence it remains to show that it is invariant under all $A_s$. But this holds since for $v\in T$, the computation
\begin{align*}
[A_i,A_j]A_s v=\sum_{l=1}^k c_{ij}^l A_l A_s v= A_s \sum_{l=1}^k c_{ij}^l A_l v=A_s[A_i,A_j]v=0
\end{align*}
shows that $A_s v\in T$. 
This concludes the proof.
\end{proof}
The Lemma is similar to the well-known Shemesh criterion for the existence of a common eigenvector of two (complex) matrices, see~\cite[Thm.~3.1]{Shemesh84}, and in general it is more efficient than the extension provided in~\cite[Thm.~1.2]{Jamiolkowski15}, which has time-complexity exponential in $k$.

\subsection*{Algorithm for the General Case} 

Using this result we can formulate an algorithm for the common eigenvector problem. 
We will assume that the matrices $A_1,\ldots,A_k$ are linearly independent, since otherwise we can simply remove the offending elements.
Now extend the list $A_1,\ldots,A_k$ by appending matrices $A_{k+1},\ldots,A_{d}$ such that the full list forms a basis of the $d$-dimensional Lie algebra generated by $A_1,\ldots,A_k$.
Importantly removing linearly dependent elements does not change the generated Lie algebra, and the common eigenvectors of the matrices $A_1,\ldots,A_k$ are exactly the same as those of the enlarged set $A_1,\ldots,A_d$.
Now compute the space $T=\bigcap_{i,j=1}^{d}\ker[A_i,A_j]$. 
Since, by part~\ref{it:T-invariant} of the Lemma, $T$ is invariant under all $A_i$ with $i\in\{1,\ldots,d\}$, we can restrict them to $T$, and it holds that $[A_i|_T,A_j|_T]=0$ for all $i,j\in\{1,\ldots,d\}$. 
Then using the algorithm for the commutative case, we can find a common eigenvector of all $A_i|_T$ for $i\in\{1,\ldots,k\}$ on $T$, and this will be the desired eigenvector. 
Note that this fails if and only if $T$ is trivial, which happens only if there is no common eigenvector by part~\ref{it:ev-in-T} of the Lemma.

Computing of order $d^2$ commutators takes time $\mc O(d^2n^3)$.
Similarly, the subspace $T$ can be found by determining the kernel of the corresponding matrix of size $n\tfrac{d(d-1)}{2} \times n$, (here we used the anti-symmetry of the commutator) which can be done in the same time-complexity.
The time-complexity of generating the Lie algebra is given in App.~\ref{app:gen-lie-alg}, and that of finding a common eigenvector on $T$ was given above (and doesn't influence the overall result).
Thus, the total complexity is $\mc O(dn^3(kn+d))$, which in terms of $n$ can be bounded by $\mc O(n^8)$ since $k,d\leq n^2$. % \mc O(dkn^4 + d^2 n^3)=

\subsection*{Algorithm for Simultaneous Triangulation} 

As an application we will consider the problem of simultaneously triangulating an arbitrary set of $n\times n$ matrices $A_1,\ldots,A_k$. We consider the following naive algorithm. First compute a common eigenvector of all $A_i$, and call it $v_1$. Change to a basis whose first vector is $v_1$ and restrict to the $(n-1)\times(n-1)$ block in the lower right. Iterate the procedure until the matrices are in triangular form. It turns out that this algorithm actually works whenever the $A_i$ are simultaneously triangulable. This follows from~\cite[Lem.~1.5.2]{Radjavi00} which states that if a collection of matrices is triangulable, then every flag of invariant subspaces is contained in a complete flag of invariant subspaces. 
More generally, this is a direct consequence of the Jordan--H{\"o}lder Theorem for modules, cf.~\cite[Thm.~3.2.1]{Hazewinkel04}.
See also~\cite{Dubi09}. %~\cite[Thm.~10.5]{Isaacs09}. 

It is clear from the above that the total time-complexity of the algorithm lies in $\mc O(n^7(k+n))$.

\section*{Acknowledgements}

I thank R.~Zeier and T.~Schulte-Herbrüggen for their valuable feedback.
The work was funded i.a.\ by the Excellence Network of Bavaria under ExQM, by {\it Munich Quantum Valley} of the Bavarian State Government with funds from Hightech Agenda {\it Bayern Plus}.

\appendix

\section{Generating a Lie Algebra} \label{app:gen-lie-alg}

A key step in our algorithm for computing a common eigenvector of a given set of matrices is that of determining a basis for the Lie algebra generated by the matrices. 
Algorithms for this task have been presented in~\cite{BW79,Isidori95,SchiFuSol01,Elliott09} in the context of control theory, but unfortunately no complexity analysis is provided.
Here we give a simple meta-algorithm which includes the standard approach and we analyze its complexity.

In this section $\K$ denotes any field. 
If $\K$ is $\R$ or $\C$, the algorithm can also be implemented numerically, i.e.~using floating point arithmetic.
However, care has to be taken here, since, for instance, two generic real matrices will generate all of $\R^{n\times n}$, see ~\cite[Sec.~1]{BW79}.

\subsection*{Meta-algorithm for Generating a Lie Algebra}

Given a linearly independent set of matrices $\mc A=\{A_1,\ldots,A_k\}$ in $\mbb K^{n\times n}$, we denote the Lie algebra generated by $\mc A$ with $\langle \mc A\rangle_{\sf Lie}$ and its dimension with $d$.
The meta-algorithm proceeds as follows:

Set $\mc A_1=\mc A$ and $i=1$.
Pick two elements $B_i,B_i'\in\mc A_i$ and compute the commutator $[B_i,B_i']$.
If the commutator does not lie in $\operatorname{span}(\mc A_i)$, then set $\mc A_{i+1}=\mc A_i \cup \{[B_i,B_i']\}$, otherwise set $\mc A_{i+1}=\mc A_i$. Increment $i$ and repeat.
Terminate the algorithm when all commutators are guaranteed to lie in the span of $\mc A_i$.

To turn this into a concrete algorithm, a procedure for choosing $B_i,B_i'$ and a termination criterion have to be provided.
Below we will show how to check for linear independence such that that each iteration can be implemented with time-complexity $\mc O(n^4)$.
The time-complexity of the entire algorithm is hence determined by the number of iterations. 
Clearly the algorithm needs at least $d-k$ iterations, but it might take much longer. 
Note that of course $d\leq n^2$.

The most naive algorithm would try all commutators in $\mc A_i$, and terminate if none of them yield a new dimension. 
This leads to a bound of $\mc O(d^3)$ on the number of iterations and a total time-complexity of $\mc O(d^3n^4)$.
By simply not computing commutators which have been computed before, and using the anti-symmetry of the commutator, one can bound the number of iterations by $\frac{d(d-1)}{2}$, and hence the total time-complexity is then $\mc O(d^2n^4)$, and thus $\mc O(n^8)$ in term of $n$.
Note that this recovers the complexity given in~\cite{ZS11}, although the algorithm is slightly different.
The algorithm presented in~\cite{Elliott09}, based on the realization that one only needs to commute new elements with the initial matrices in $\mc A$, yields a bound of $\mc O(dk)$ on the number of iterations and a total time-complexity of $\mc O(dkn^4)$. 
In terms of $n$ alone we again get a bound of $\mc O(n^8)$.

\subsection*{Update Algorithm for Linear Independence}

Consider a matrix $M\in\mbb K^{r\times s}$ with $s\leq r$. 
Determining if the columns of this matrix are linearly independent can be done using Gaussian elimination in time-complexity $\mc O(rs^2)$.
In our Lie algebra generation algorithm above however we find ourselves in the following scenario.
We obtain one vector in $\mbb K^r$ after another.
We keep the first non-zero vector and discard any zero vectors.
Then for every following vector we keep it only if it is independent from all previously kept ones, otherwise we discard it.
Using again Gaussian elimination, this can be implemented in $\mc O(r^2)$ for each vector. The total runtime for $s$ vectors will be slightly worse with $\mc O(r^2s)$, but now we can give a result after each vector.

The algorithm proceeds as follows:
For the first non-zero vector perform Gaussian elimination to obtain the first standard basis vector and store the row operations performed in a matrix $R_1$.
For every following vector, continue the Gaussian elimination, while updating also the matrix $R_i$ storing the row operations.
If we detect linear dependence, discard the vector and revert the matrix $R_i$ to its previous state.

%\printbibliography
\bibliographystyle{acm}
\bibliography{./general.bib}

\end{document}